\documentclass[a4paper,10pt]{amsart}
\usepackage{amssymb,amsthm,amsmath,fixmath }
\usepackage{ifthen}
\usepackage{graphicx}
\nonstopmode
 \numberwithin{equation}{section}
\newtheorem{thm}{Theorem}[section]
\newtheorem{cor}[thm]{Corollary}
\newtheorem{lem}[thm]{Lemma}
\newtheorem{prop}[thm]{Proposition}
\newtheorem{defn}[thm]{Definition}

\theoremstyle{definition}
\newtheorem{rmk}[thm]{Remark}

{\qed\bigskip}

\newcounter{alphabet}
\newcounter{tmp}

\ifx\undefined\bysame
\newcommand{\bysame}{\leavevmode\hbox to3em{\hrulefill}\,}
\fi

\begin{document}
\baselineskip=21pt
\markboth{} {}

\bibliographystyle{amsplain}
\title[Hilbert space valued Gabor frames in Weighted Amalgam Spaces]
{Hilbert space valued Gabor frames in Weighted Amalgam Spaces}

\author{Anirudha poria}
\author{Jitendriya Swain}
\address{Department of Mathematics,
Indian Institute of Technology Guwahati,
Guwahati 781039, \;\; India.} 
\email{a.poria@iitg.ac.in, jitumath@iitg.ac.in}

\begin{abstract}
 Let $\mathbb{H}$ be a separable Hilbert space. In this paper we establish a generalization of Walnut's representation and Janssen's representation of the $\mathbb{H}-$valued Gabor frame operator on $\mathbb{H}-$valued weighted amalgam spaces $W_{\mathbb{H}}(L^p,L^q_v)$, $1 \leq p, q \leq \infty$. Also we show that the frame operator is invertible on $W_{\mathbb{H}}(L^p,L^q_v)$, $1 \leq p, q \leq \infty$, if the window function is in the Wiener amalgam space $W_{\mathbb{H}}(L^{\infty},L^1_w)$. Further, we obtain the Walnut representation and invertibility of the frame operator corresponding to Gabor superframes and multi-window Gabor frames on $W_{\mathbb{H}}(L^p,L^q_v)$, $1 \leq p, q \leq \infty,$ as a special case by choosing the appropriate Hilbert space $\mathbb{H}$.

\end{abstract}
\date{\today}
\maketitle
\def\BC{{\mathbb C}} \def\BQ{{\mathbb Q}}
\def\BR{{\mathbb R}} \def\BI{{\mathbb I}}
\def\BZ{{\mathbb Z}} \def\BD{{\mathbb D}}
\def\BP{{\mathbb P}} \def\BB{{\mathbb B}}
\def\BS{{\mathbb S}} \def\BH{{\mathbb H}}
\def\BE{{\mathbb E}}
\def\BN{{\mathbb N}}
\def\LP{{W(L^p(\BR^d, \BH), L^q_v)}}
\def\LPN{{W_{\BH}(L^p, L^q_v)}}
\def\LPQ{{W_{\BH}(L^{p'}, L^{q'}_{1/v})}}
\def\L1{{W_{\BH}(L^{\infty}, L^1_w)}}
\def\LB{{L^p(Q_{1/ \beta}, \BH)}}
\def\SP{S^{p,q}_{\tilde{v}}(\BH)}
\def\f{{\bf f}}
\def\h{{\bf h}}
\def\hp{{\bf h'}}
\def\m{{\bf m}}
\def\g{{\bf g}}
\def\ga{{\boldsymbol{\gamma}}}
\vspace{-.5cm}

\section{Introduction}
For $\alpha,\beta>0$, $g \in L^2(\BR^d)$ and $n,k \in \BZ^d$ define $M_{\beta n} g(x):= e^{2 \pi i \langle \beta n , x \rangle} g(x)$ and $T_{\alpha k}g(x):=g(x-\alpha k)$. The collection of functions $\mathcal{G}(g, \alpha,\beta)=\{ M_{\beta n}T_{\alpha k}g : \; k,n \in \BZ^d\}$ in $L^2(\BR^d)$ is called a \textit{Gabor frame} or a \textit{Weyl-Heisenberg frame} if there exist constants $A, B >0$ such that
\begin{equation}
A \Vert f \Vert^2_2 \leq \sum_{k,n \in \BZ^d} |\langle f, M_{\beta n}T_{\alpha k}g \rangle |^2 \leq B \Vert f \Vert^2_2, \;\; \forall f \in L^2(\BR^d).
\end{equation}
The associated frame operator called the Gabor frame operator has the form 
\begin{equation}\label{eq01}
S_{g}f:=\sum\limits_{k,n \in \BZ^d} \langle f, M_{\beta n}T_{\alpha k}g \rangle M_{\beta n}T_{\alpha k}g, \;\; f \in L^2(\BR^d).
\end{equation}

If $g \in L^2(\BR^d)$ generates a Gabor frame $\mathcal{G}(g, \alpha,\beta)$ then there exists a dual window (called the canonical dual window) $\gamma=S_{g}^{-1}(g) \in L^2(\BR^d)$ such that $\mathcal{G}(\gamma, \alpha,\beta)=\{M_{\beta n}T_{\alpha k} \gamma: k,n \in \BZ^d \}$ is also a frame for $L^2(\BR^d)$, called the canonical dual Gabor frame. Consequently every $f \in L^2(\BR^d)$ possess the expansion
\begin{equation}\label{eq35}
f=\sum_{k,n \in \BZ^d}\langle f, M_{\beta n}T_{\alpha k}g \rangle M_{\beta n}T_{\alpha k}\gamma  = \sum_{k,n \in \BZ^d}\langle f, M_{\beta n}T_{\alpha k}\gamma \rangle M_{\beta n}T_{\alpha k}g
\end{equation}
 with unconditional convergence in $L^2(\BR^d)$.

 The convergence of the above expansion can be extended to $L^p$ spaces under additional assumptions on $g$ and $\gamma$ (see \cite{gro, hei}).
If $g$, $\gamma$ are in Feichtinger's algebra then (\ref{eq35}) holds for modulation spaces (see \cite{zim, gro}), if $g$, $\gamma$ are in the Wiener algebra $W(L^{\infty}, L^1)(\BR^d)$ then it holds for $L^p(\BR^d)$ and Wiener amalgam spaces (see \cite{fei, hei, oko}), and if $g, \gamma \in W(C_r,L^1)(\BR^d)$ then it holds for local Hardy spaces $h_p(\BR^d)$, $p>d/(d+r)$ (see \cite{lak, wei}), where $C_r(\BR^d)$ denotes the Lipschitz or H\"older space with $0 < r \leq 1$.
Moreover, it is desirable to find a Gabor frame such that the generator $g$ and its canonical dual $\gamma$ have the similar properties (viz. smoothness or decay) in applications (see \cite{str}). In this direction the following results are known: If $g$ has compact support then in general $\gamma$ is no longer compactly supported but has exponential decay \cite{bol} and in \cite{pre} Del Prete proved that if $g$ has exponential decay, then $\gamma$ also has exponential decay. If $g$ can be estimated by $C(1+|t|)^{-s}$, then the same holds for $\gamma$ (see \cite{str1}). It is natural to ask  if $g$ is in a given function space whether its canonical dual is in the same space? The first result in this direction is due to Janssen \cite{jan} which ensures that if $g$ is in the Schwartz space on $\BR^d$ then its canonical dual is in the same space. Gr\"ochenig and Leinert \cite{g} proved that if $g$ is an element in the Feichtinger's  algebra and $\mathcal{G}(g, \alpha,\beta)$ is a Gabor frame for $L^2(\mathbb{R}^d)$,  then  the canonical dual $\gamma$ is in the same space. In \cite{kri}, the authors prove that if $\mathcal{G}(g, \alpha,\beta)$ is a Gabor frame generated by $g\in W(L^{\infty}, L^1_v)$ then the Gabor frame operator is invertible on the Wiener amalgam spaces $W(L^\infty,L^1_v)$, where $v$ is an admissible weight function. In \cite{weis}, Weisz extended the analogous result on amalgam spaces $W(L^\infty,L^q_v)$ for $1\leq q \leq 2.$ The above results were based on the reformulation of a non-commutative Wiener's lemma proved by Baskakov \cite{bas1,bas2}. In \cite{bal3}, Balan et al. obtained a similar result for multi-window Gabor frames on Wiener amalgam spaces using a recent Wiener type result on non-commutative almost periodic Fourier series \cite{bal4}.
The vector-valued Gabor frames or superframes  were introduced by Balan \cite{bal1} in the context of ``multiplexing",  and several well-known results for
Gabor frames are extended to superframes in \cite{bal2,han}.
  The superwavelet and Gabor frames in $L^2(\mathbb{R}^d,\mathbb{C}^n)$ is widely applicable in mathematics and several branches of engineering (see \cite{bal1,bal2,dut1,dut2,lyu,gu,han,li,zib1,zib2}). Using the growth estimates for the Weierstrass $\sigma$-function and a new type of interpolation problem for entire functions on Bargmann-Fock space, Gr\"ochenig and Lyubarskii \cite{lyu} obtained a complete characterization of all lattices $\Lambda\subset \mathbb{R}^2$ such that the Gabor system with first $n+1$ Hermite functions (generators) forms a frame for $ L^2(\BR, \BC^n)$.

The main objective of this article is to investigate the following for superframes on vector-valued amalgam spaces: 
  \begin{enumerate}
  \item Walnut's representation for the superframe operator on vector valued amalgam spaces.
  \item Convergence of Gabor expansions on vector valued amalgam spaces.
    \item Invertibility of the superframe operator on vector valued amalgam spaces.
      \end{enumerate}

 We consider the separable Hilbert space valued Gabor frames and perform Gabor analysis on $W_{\mathbb{H}}(L^p,L^q_v), 1\leq p,q\leq\infty$ to obtain (1), (2) and (3) for Gabor superframes and multi-window Gabor frames as a special case by choosing the appropriate separable Hilbert space $\mathbb{H}$ (see Remark \ref{r} and Remark \ref{r1}).
We start with the definition of $\mathbb{H}-$valued Gabor frames on $L^2(\mathbb{R}^d,\mathbb{H})$.

\begin{defn}\label{frame}
 Let $\alpha, \beta >0$ and $\g \in L^2(\BR^d,\BH)$ be given. For $n,k\in\BZ^d$, define $M_{\beta n} \g(x)= e^{2 \pi i \langle \beta n , x \rangle} \g(x)$ and $T_{\alpha k}\g(x)=\g(x-\alpha k)$. The $\BH$-valued Gabor system $\mathcal{G}(\g,\alpha, \beta )$ is a frame for $L^2(\BR^d, \BH)$ if there exist constants $A,B > 0$ such that for all $\f \in L^2(\BR^d, \BH)$,
\begin{equation}
A \Vert \f \Vert^2_{L^2(\BR^d, \BH)} \leq \sum_{k,n \in \BZ^d} |\langle \f, M_{\beta n}T_{\alpha k}\g \rangle_{L^2(\BR^d, \BH)} |^2 \leq B \Vert \f \Vert^2_{L^2(\BR^d, \BH)}.
\end{equation}\end{defn}
 For $\g,\ga\in L^2(\mathbb{R}^d,\BH)$, the associated frame operator is given by $$S_{\g,\ga}\f=\sum\limits_{k,n \in \BZ^d} \langle \f, M_{\beta n}T_{\alpha k}\g \rangle M_{\beta n}T_{\alpha k}\ga, \;\; \f \in L^2(\BR^d,\mathbb{H}).$$
  Moreover if $\g,\ga\in W_{\mathbb{H}}(L^\infty,L^1_v), 1\leq p,q\leq\infty,$ then we obtain Walnut's representation of the frame operator on $\BH-$valued amalgam spaces (see Theorem \ref{th1}):
\begin{equation}\label{eq00}
 S_{\g,\ga}\f(x)=\beta^{-d} \sum_{n \in \BZ^d} G_n(x)\left( T_{\frac{n}{\beta}} \f(x) \right),
 \end{equation}
where $G_n(x)$ is an element of $B(\BH)$, the class of all bounded linear operators on $\BH$. Since we deal with vector-valued functions, obtaining the above expression is bit technical. The expression in Walnut's representation of the frame operator in (\ref{eq00}) is similar to scalar-valued case (see \cite{oko}), but has a different meaning.

Further, we show that if the window function $\g\in W_\BH(C_0,L^1_w)$, the subspace formed by the functions of $\BH-$valued Wiener space that are continuous, then the canonical dual is in the same space. To obtain this result we show that the frame operator $S_\g ~(=S_{\g,\g})$ is invertible on $\LPN, 1 \leq p,q \leq \infty$ and $S_{\g}^{-1}: W(L^{p}, L^{q}_{v}) \rightarrow W(L^{p}, L^{q}_{v})$ is continuous both in $\sigma(W(L^{p}, L^{q}_{v}), W(L^{p'}, L^{q'}_{1/v}))$ and the norm topologies. Such invertibility results is addressed for modulation spaces (see \cite{gro2, g}) and for $L^p$ spaces (see \cite{kri}) by using Wiener's $1/f$ lemma (see \cite{bal4, bas1, bas2}). We construct a Banach algebra of operators admitting an expansion like (\ref{eq00}) and use a version of Wiener's $1/f$ lemma (see \cite{bal4}) to prove that the algebra is spectral within the class of bounded linear operators on $L^p(\BR^d,\BH),\; 1 \leq p \leq \infty$.

This paper is organized as follows. In Section 2, we provide basic background and define $\BH-$valued amalgam spaces. In Section 3, we derive Walnut's representation for the $\BH$-valued Gabor frame operator and discuss the convergence of $\BH$-valued Gabor expansions. In Section 4, we prove the spectral invariance theorem for a subalgebra of weighted-shift operators in $B(L^p(\BR^d, \BH)),$ and finally in section 5 we prove the invertibility of the frame operator $S_\g$ on $\LPN, 1 \leq p,q \leq \infty$.

\section{Notations and Background}
Throughout this paper, let $\BH$ denote a separable complex Hilbert space. Let $Q_{\alpha}$ denote the cube $Q_{\alpha}=[0,\alpha)^d$ and $\chi_E$ is the characteristic function of a measurable set $E$.

\begin{defn}
For $1 \leq p \leq \infty$ and a strictly positive function $w$ on $\BR^d$, let $L^p_w(\BR^d, \BH)$ denote the space of all equivalence classes of $\BH$-valued Bochner integrable functions $\f$ defined on $\BR^d$ with $\displaystyle\int_{\BR^d} \Vert \f(x) \Vert^p_{\BH}\, w(x)^p \,dx < \infty$, with the usual adjustment if $p=\infty.$
\end{defn}
\begin{defn}
The Fourier transform of $\f \in L^1(\BR^d, \BH)$ is
 \[\hat{\f}(w)=\mathcal{F}\f(w)= \int_{\BR^d} \f(t) \,e^{-2 \pi i \langle w,t \rangle} \,dt, \;\;\; w \in \BR^d. \]
\end{defn}

Let $\f \in L^{p'}(\BR^d,\BH)$. Define $\Lambda_{\f}:L^{p}(\BR^d,\BH) \rightarrow \BC$ by $\Lambda_{\f}(\g)=\displaystyle\int_{\BR^d}\langle \f(x), \g(x)\rangle_{\BH}\, dx.$ Then the map $\f \mapsto \Lambda_{\f}$ defines an isometric isomorphism of $L^{p'}(\BR^d,\BH)$ onto $[L^{p}(\BR^d,\BH)]^*$. For a more detailed study of vector valued functions we refer to Diestel and Uhl \cite{die}.
\begin{defn}
For $\f,\g \in L^2(\BR^d,\BH)$ the inner product on $L^2(\BR^d,\BH)$ is defined by \[\langle \f, \g \rangle_{L^2(\BR^d,\BH)}=\int_{\BR^d}\langle \f(x), \g(x)\rangle_{\BH}\, dx. \]
\end{defn}
\begin{defn}\label{dot}
If $x,y \in \BH$, the operator $x \odot y :\BH \rightarrow \BH$ defined by \[ (x \odot y)(z)=\langle z,y \rangle x ,\;\; z \in \BH. \]
\end{defn}
Clearly for  non-zero $x$ and $y$, $x \odot y$ is a rank one operator with  $\Vert x \odot y \Vert = \Vert x \Vert \Vert y \Vert$.

%
%
\subsection{Gabor frames in $L^2(\BR^d,\BH)$}
Assume that the Gabor system $\mathcal{G}(\g,\alpha,\beta)$ is a frame for $L^2(\BR^d,\BH)$ with frame bounds $A,B$ and let $\h$ be a vector in $\BH$ with unit norm. Then the \emph{analysis ~operator} is a bounded mapping $C_{\g,\h}:L^2(\BR^d,\BH) \rightarrow \ell^2(\BZ^{2d})$ defined by $C_{\g,\h}\f=(\langle \f,M_{\beta n}T_{\alpha k}\g \rangle \h)_{k,n \in \BZ^{d}} $. The adjoint operator of the analysis operator, called the \emph{synthesis~ operator} is a bounded mapping $ R_{\g,\h} :\ell^2(\BZ^{2d},\BH) \rightarrow  L^2(\BR^d,\BH)$ defined by $R_{\g,\h}d = \sum\limits_{k,n \in \BZ^d} \langle d_{kn}, \h \rangle_{\BH}\, M_{\beta n}T_{\alpha k}\g$. The series defining $R_{\g,\h}d$ converges unconditionally in $L^2$ for every $d \in \ell^2.$ The composition operator of analysis and synthesis operator is called the \emph{frame operator} $S_{\g}=R_{\g,\h} C_{\g,\h}:L^2(\BR^d,\BH) \rightarrow L^2(\BR^d,\BH)$, defined by $$S_{\g}\f=R_{\g,\h} C_{\g,\h}\f=\sum_{k,n \in \BZ^d}\langle \f,M_{\beta n}T_{\alpha k} \g \rangle M_{\beta n}T_{\alpha k} \g.$$ The frame operator is strictly positive, invertible, self-adjoint satisfying $AI_{L^2(\BR^d,\BH)} \leq S_{\g} \leq BI_{L^2(\BR^d,\BH)}$ and $B^{-1}I_{L^2(\BR^d,\BH)} \leq S^{-1}_{\g} \leq A^{-1}I_{L^2(\BR^d,\BH)}$ .
For $\ga=S_{\g}^{-1}\g\in L^2(\BR^d,\BH)$, the Gabor expansions
\begin{equation}\label{*}
R_{\ga,\h} C_{\g,\h}\f=\sum_{k,n \in \BZ^d}\langle \f,M_{\beta n}T_{\alpha k} \ga \rangle M_{\beta n}T_{\alpha k} \g,
\end{equation}
converge to $f$ unconditionally in $L^2(\BR^d,\BH)$.
\subsection{Weight functions}
Weight functions play an important role in time frequency analysis and occur in many problems and contexts. For our study we need the following types of weights.

A function $ w :\BR^d \rightarrow (0, + \infty)$ is called a weight if it is continuous and symmetric (i.e. $w(x)=w(-x))$. A weight $w$ is submultiplicative if $ w(x+y) \leq w(x)w(y), \;\; x,y \in \BR^d.$ Given a submultiplicative weight $w$, a second weight $ v :\BR^d \rightarrow (0, + \infty)$ is called $w$-moderate if there exists a constant $C_v > 0$ such that
\begin{equation} \label{eqw}
v(x+y) \leq C_v w(x)v(y), \;\; x,y \in \BR^d.
\end{equation}
A weight $w$ is called admissible if it is submultiplicative, $w(0)=1$ and $w$ satisfies the Gelfand-Raikov-Shilov condition namely $\lim_{k \rightarrow \infty} w(kx)^{1/k}=1, \; x \in \BR^d.$

By (\ref{eqw}) and by symmetry of $w$ one can conclude that the class of $w$-moderate weights is closed under reciprocals, and consequently the class of spaces $L^p_v$ using $w$-moderate weights is closed under duality (with the usual exception for $p=\infty$). Also  $L^p_v(\BR^d, \BH)$ is translation-invariant for moderate weights like  $L^p_v(\BR^d)$ (see Proposition 11.2.4. of \cite{hei1}).
%
Throughout this paper, $w$ will denote a submultiplicative weight function and $v$ will denote an $w$-moderate function. Given an $w$-moderate weight $v$ on $\BR^d$, we will often use the notation $\tilde{v}$ to denote the weight on $\BZ^d$ defined by $\tilde{v}(k)=v(\alpha k)$, and for a weight $v$ on $\BR^{2d}$ we define $\tilde{v}(k,n) = v(\alpha k, \beta n).$
\subsection{Amalgam spaces}
\begin{defn}
Given an $w$-moderate weight $v$ on $\BR^d$ and given $1 \leq p, q \leq \infty $, the weighted amalgam space $\LP$ is the Banach space of all $\BH$-valued measurable functions $\f : \BR^d \rightarrow \BH$ for which the norm
\begin{equation} \label{eqa}
\Vert \f \Vert_{\LP} := \left(\sum_{k \in \BZ^d} \Vert \f \cdot T_{\alpha k} \chi_{Q_{\alpha}} \Vert_{L^p(\BR^d, \BH)}^q {v(\alpha k)}^q \right)^{1/q}  < \infty,
\end{equation}
with obvious modification for $q= \infty .$
\end{defn}
Throughout this paper we denote the weighted amalgam space as $W_{\BH}(L^p, L^q_v)$ instead of $W(L^p(\BR^d, \BH), L^q_v)$. We refer to \cite{fei1, fei2, fei3, fou, hei1} for discussions of weighted amalgam spaces and their applications. The space $\LPN$ is independent of the value of $\alpha$ used in (\ref{eqa}) in the sense that each different choice of $\alpha$ yields an equivalent norm for $\LPN$ (as in the scalar valued case).

Notice that the space $\LPN$ is a Banach function space in the sense of \cite{ben}. The  K\"othe dual or associated space (see \cite{ben}) of $\LPN$ is the space of all measurable functions
$\g: \BR^d \rightarrow \BH $ such that $\langle \f(x), \g(x) \rangle_{\BH}\in L^1(\BR^d),$ for each $\f\in\LPN$. By Theorem 2.9 of \cite{ben}, the  K\"othe dual of $\LPN$ coincides with $W_{\BH}(L^{p'}, L^{q'}_{1/v}),$ where $1/p+1/p'=1/q+1/q'=1$ for all $1 \leq p, q \leq \infty $. A series $\sum_{k\in J}\f_k$ converges (unconditionally) in the topology of $\sigma(\LPN, W_{\BH}(L^{p'}, L^{q'}_{1/v})$ if the series $\sum_{k\in J}\langle\f_k,\g\rangle$ converges (is independent of the ordering of $J$) for each $\g\in W_{\BH}(L^{p'}, L^{q'}_{1/v}).$

Let $\f \in W_{\BH}(L^{p'}, L^{q'}_{1/v})$. Define $\Theta_{\f}:\LPN \rightarrow \BC$ by $\Theta_{\f}(\g)=\langle \f, \g \rangle= \int_{\BR^d} \langle \f(x), \g(x) \rangle_{\BH} dx.$ Then the map $\f \mapsto \Theta_{\f}$ defines an isometric isomorphism of $\LPN^*$ onto $W_{\BH}(L^{p'}, L^{q'}_{1/v})$. We refer to Theorem 11.7.1 of \cite{hei1} for the duality result in scalar valued amalgam spaces. We summarize the above discussions in the following lemma.
\begin{lem}
Let $v$ be an $w$-moderate weight and $1/p+1/p'=1/q+1/q'=1$. For $1 \leq p, q < \infty,$ the dual space of $\LPN$ is $ W_{\BH}(L^{p'}, L^{q'}_{1/v})$ and the K\"othe dual of $\LPN$ is $ W_{\BH}(L^{p'}, L^{q'}_{1/v}).$
\end{lem}
\section{$\BH$-valued Gabor Expansions in weighted amalgam spaces}
The theory of Gabor expansions on Wiener amalgam spaces has been discussed in \cite{fei, gra, hei, oko} for a separable lattice $\Lambda=\alpha \BZ^d \times \beta \BZ^d, \; \alpha, \beta >0$. In this section we generalize and extend the Walnut's representation of the $\BH$-valued frame operator and show the convergence of $\BH$-valued Gabor expansions on $\LPN$.
\begin{defn}
The Fourier transform of $\f \in L^1(Q_{1/ \beta}, \BH)$ is the sequence $\hat{\f}$ defined by \[ \hat{\f}(n)= \mathcal{F}\f(n)=\beta^d \int_{Q_{1/ \beta}}e^{-2 \pi i \beta \langle n,t \rangle} \f(t)\,  dt, \;\;\;\; n \in \BZ^d . \]
\end{defn}
For $1 \leq p,q \leq \infty $, denote $\mathcal{F}L^p(Q_{1/ \beta}, \BH)$ by the image of $L^p(Q_{1/ \beta}, \BH)$ under the Fourier transform. From the uniqueness of Fourier coefficients for functions in $L^p(Q_{1/ \beta}, \BH)$, there exists a unique function $\m \in L^p(Q_{1/ \beta}, \BH) $ such that $\hat{\m}(n)=d_n$ for every $n$, if $d=(d_n)_{n \in \BZ^d} \in \mathcal{F}L^p(Q_{1/ \beta}, \BH).$  The norm on $\mathcal{F}L^p(Q_{1/ \beta}, \BH)$ is defined by
$\Vert d \Vert_{\mathcal{F}L^p(Q_{1/ \beta}, \BH)}=\Vert \m \Vert_{p,Q_{1/ \beta}}.$
\begin{defn}\label{s}
Let $\alpha, \beta > 0$ be given. Let $$\SP= \{d=(d_{kn})_{k,n \in \BZ^d}: ~\mbox{for~ each}~ k \in \BZ^d,~\exists~ \m_k \in L^p(Q_{1/ \beta}, \BH)  ~\mbox{such that}~ \hat{\m_k}(n)=d_{kn} \},$$ and $\Vert d \Vert_{\SP}=\left( \sum_{k \in \BZ^d} \Vert \m_k \Vert^q_{p,Q_{1/ \beta}} \tilde{v}(k)^q \right)^{1/q} \;\; < \infty,$ with the usual change if $q=\infty$.
\end{defn}
Since $\BH$ is separable, $\BH$ isometrically isomorphic to $\ell^2.$ Therefore $\|\f\|_{L^p(Q_{1/ \beta}, \ell^2(\BZ))}\asymp\|\f\|_{L^p(Q_{1/ \beta}, \BH)}$.
When $1<p< \infty,$ we can write $\m_k$ as a Fourier series
\begin{equation} \label{ew4}
\m_k(x)=\sum_{n \in \BZ^d} d_{kn} e^{2 \pi i \beta \langle n,x \rangle},
\end{equation}
in the sense that the square partial sums of (\ref{ew4}) converge to $\m_k$ in the norm of $L^p(Q_{1/ \beta}, \BH)$, cf. \cite{ru} (\cite{kat}, \cite{zyg} for scalar valued case). Hence, for $1<p< \infty$ and $1\leq q < \infty $ we can write the norm on $\SP$ as
 $\Vert d \Vert_{\SP}= \left( \sum_{k \in \BZ^d}\left( \int_{Q_{1/ \beta}} \Vert \sum_{n \in \BZ^d} d_{kn} e^{2 \pi i \beta \langle n, x \rangle } \Vert^p_{\BH}  dx \right)^{q/p} \tilde{v}(k)^q \right)^{1/q}.$

The analysis and synthesis operators associated with the $\BH$-valued Gabor frame are defined as follows: Take $\h \in \BH$ such that $\Vert \h \Vert_{\BH}=1$. Let $\alpha, \beta >0$, $1 \leq p , q \leq \infty$ and fix $\g,\ga \in \L1$. For $\f \in \LPN$, define the analysis operator by
\begin{eqnarray*}
C_{\g, \h}f(k,n) &=& \langle \f,M_{\beta n}T_{\alpha k}\g \rangle \h \\
&=& \int_{\BR^d} \langle \f(x),M_{\beta n}T_{\alpha k}\g(x) \rangle_{\BH} \h dx \\
&=& \int_{\BR^d} \langle \f(x),T_{\alpha k}\g(x) \rangle_{\BH} \h e^{-2 \pi i \beta \langle n,x \rangle} dx \\
&=& \mathcal{F}((\h \odot T_{\alpha k}\g)\f)(\beta n),
\end{eqnarray*}where $\odot$ is defined in definition \ref{dot}.
So $\sum_{n \in \BZ^d} \langle \f,M_{\beta n}T_{\alpha k}\g \rangle \h e^{2 \pi i \beta \langle n, x\rangle} = \sum_{n \in \BZ^d} \mathcal{F}((\h \odot T_{\alpha k}\g)\f)(\beta n) e^{2 \pi i \beta \langle n, x\rangle}.$ Applying Poisson summation formula,
$
\sum_{n \in \BZ^d} \langle \f,M_{\beta n}T_{\alpha k}\g \rangle \h e^{2 \pi i \beta \langle n, x\rangle} = \beta^{-d} \sum_{n \in \BZ^d}((\h \odot T_{\alpha k}\g(x))\f)(x-\frac{n}{\beta})
= \m_k(x)~(say).
$

For $d=(d_{kn}) \in \SP$, define the synthesis operator by
\begin{eqnarray}\label{eqr}
R_{\g,\h}d(x) &=& \sum_{k,n \in \BZ^d} \langle d_{kn}, \h \rangle_{\BH}\,M_{\beta n}T_{\alpha k}\g(x) \nonumber \\
&=& \sum_{k \in \BZ^d} \left\langle \sum_{n \in \BZ^d}d_{kn} e^{2 \pi i \beta \langle n,x \rangle}, \h \right\rangle_{\BH}\, T_{\alpha k}\g(x)\\
&=& \sum_{k \in \BZ^d} \langle \m_k(x), \h \rangle_{\BH}\, T_{\alpha k}\g(x) \nonumber
\end{eqnarray}
where $\m_k(x)=\sum\limits_{n \in \BZ^d}d_{kn} e^{2 \pi i \beta \langle n,x \rangle}$ is $1/\beta$-periodic.

From the above observations we obtain the analogue of Walnut's representation for the $\BH$-valued Gabor frames in the following theorem.
\begin{thm}\label{th1}
Let $\alpha,\beta>0$ and $1 \leq p,q \leq \infty.$ Let $v$ be an $w$-moderate weight on $\BR^d$ and $\g, \ga \in \L1$ with $\h \in \BH$ be given. If $ \| \h \|_{\BH}=1$ then
\begin{enumerate}\item[(a)] The analysis operator $C_{\g,\h}: \LPN\rightarrow \SP $ defined by $C_{\g,\h}\f=(\langle \f,M_{\beta n}T_{\alpha k}\g \rangle \h)_{k,n \in \BZ^d}$ is a bounded. There exist unique functions $\m_k \in \LB$ satisfying $\hat{\m}_k(n)=C_{\g,\h}\f(k,n)$ for all $k,n \in \BZ^d$ which can be expressed explicitly by
\begin{equation} \label{eq1}
\m_k(x)=\beta^{-d} \sum_{n \in \BZ^d} (\h \odot T_{\alpha k}\g(x))\f(x-\frac{n}{\beta})=\beta^{-d} \sum_{n \in \BZ^d} (\h \odot T_{\alpha k +\frac{n}{\beta} }\g(x))T_{\frac{n}{\beta}}\f(x)
\end{equation}
The series (\ref{eq1}) converges unconditionally in $\LB$ $($unconditionally in the topology of $\sigma(L^{\infty}(Q_{1/ \beta}, \BH),L^1(Q_{1/ \beta}, \BH))$ if $p=\infty ).$
\item[(b)] The analysis operator $ R_{\g,\h} : \SP \rightarrow \LPN $ defined by \begin{equation} \label{eq2}
R_{\g,\h}d=\sum_{k \in \BZ^d} \langle \m_k(\cdot),\h \rangle_{\BH} T_{ak}\g
\end{equation} is bounded, where $d \in \SP$, $\m_k \in L^p(Q_{\alpha} , \BH)$ be the unique functions satisfying $\hat{\m}_k(n)=d_{kn}$ for all $k,n \in \BZ^d$. The series (\ref{eq2}) converges unconditionally in $\LPN$ $($unconditionally in the $\sigma(\LPN,\LPQ)$ topology if $p=\infty$ or $q=\infty).$
%
%
\item[(c)] The frame operator $R_{\ga,\h}C_{\g,\h}$ admits Walnut's representation on $\LPN$:
\begin{equation} \label{eq3}
R_{\ga,\h}C_{\g,\h}\f=\beta^{-d} \sum_{n \in \BZ^d} G_n\left( T_{\frac{n}{\beta}} \f \right)
\end{equation}
holds for $\f \in \LPN$, with the series (\ref{eq3}) converging absolutely in $\LPN$, where
\begin{equation}\label{eqg}
G_n(x)=\sum_{k \in \BZ^d} T_{\alpha k} \ga(x) \odot T_{\alpha k + \frac{n}{\beta}} \g(x) \in B(\BH).
\end{equation}
\end{enumerate}
\end{thm}
We need the following lemma to prove Theorem \ref{th1}.
 \begin{lem} \label{lem4}
Let $w$ be a submultiplicative weight, and let $\alpha, \beta > 0$ be given. Then there exists a constant $C=C(\alpha, \beta, w)>0$ such that if $\g, \ga \in \L1 $ and the functions $G_n$ are defined by (\ref{eqg}), then $ \sum\limits_{n \in \BZ^d} \Vert G_n \Vert_{L^{\infty}(\BR^d, B(\BH))} w\left(\frac{n}{\beta}\right) \leq C \; \Vert \g \Vert_{\L1}  \Vert \ga \Vert_{\L1}. $
\end{lem}
\begin{proof}
Using the fact that $w$ is $w$-moderate and the translation invariance property for moderate weights (Proposition 11.2.4. of \cite{hei1}) we get $\|fw\|_{W_{\BH}(L^{\infty}, L^1)}\asymp\|f\|_{\BH(L^{\infty}, L_w^1)}.$ Using the similar idea as in Lemma 6.3.1 of \cite{gro} we get the above inequality.
\end{proof}

\begin{rmk}\label{rm1}Let $G=\{G_n\}$. Then by the above lemma $G\in\ell^1(\mathbb{Z}^d,L^\infty(\mathbb{R}^d,B(H)))$.
Then $l$-th Fourier coefficient of $G_n$ is
\begin{eqnarray*}
\hat{G}_n(l) &=& \alpha^{-d}\int_{Q_{\alpha}} G_n(x)e^{- 2 \pi i \langle l, x/\alpha \rangle} dx\\
&=& \alpha^{-d}\int_{Q_{\alpha}} (\sum_{k \in \BZ^d} T_{ak}\ga(x) \odot T_{\alpha k +\frac{n}{\beta} }\g(x))e^{- 2 \pi i \langle l, x/\alpha \rangle} dx\\
&=& \alpha^{-d}\int_{\BR^d} (\ga(x)\odot T_{\frac{n}{\beta}}\g(x)) e^{- 2 \pi i \langle l, x/\alpha \rangle} dx\\
&=& \alpha^{-d}\int_{\BR^d} \ga(x)\odot M_{\frac{l}{\alpha}} T_{\frac{n}{\beta}}\g(x) dx \\
&:=& \alpha^{-d} [\ga,M_{\frac{l}{\alpha}} T_{\frac{n}{\beta}}\g]
\end{eqnarray*}
Then Fourier series
\begin{equation}\label{eq33}
G_n(x)=\alpha^{-d} \sum\limits_{l \in \BZ^d} [\ga,M_{\frac{l}{\alpha}} T_{\frac{n}{\beta}}\g] e^{2 \pi i \langle l, x/\alpha \rangle}
\end{equation}
is convergent in $L^2(Q_{\alpha},B(\BH))$. By substituting this into Walnut's representation, we obtain the expression
\begin{eqnarray*}
S_{\g,\ga}\f=\beta^{-d}\sum_{n \in \BZ^d} G_n \left(T_{\frac{n}{\beta}}\f \right)=(\alpha \beta)^{-d}\sum_{n \in \BZ^d} \left(\sum_{l \in \BZ^d} [\ga,M_{\frac{l}{\alpha}} T_{\frac{n}{\beta}}\g] \right) \left( M_{\frac{l}{\alpha}}T_{\frac{n}{\beta}}\f \right)
\end{eqnarray*}
or in operator notation,
\begin{eqnarray}
S_{\g,\ga}=(\alpha \beta)^{-d}\sum_{n \in \BZ^d}\left( \sum_{l \in \BZ^d}  [\ga,M_{\frac{l}{\alpha}} T_{\frac{n}{\beta}}\g]\right) \left( M_{\frac{l}{\alpha}}T_{\frac{n}{\beta}} \right)
\end{eqnarray}
\end{rmk}
This is $\BH$-valued analogue of {\em Janssen's representation} for the $\BH$-valued frame operator $S_{\g,\ga}$. Using Janssen's representation we obtain the $\BH$-valued analogue of Wexler-Raz biorthogonality relation in the following theorem.
\begin{thm}\label{th9}(Wexler-Raz biorthogonality relation)
Assume that $\mathcal{G}(\g,\alpha,\beta), \; \mathcal{G}(\ga,\alpha,\beta)$ are Bessel sequence in $L^2(\BR^d,\BH)$. Then the following conditions are equivalent:\\
$(i)$ $S_{\g,\ga}=S_{\ga,\g}=I$ on $L^2(\BR^d,\BH)$.\\
$(ii)$ $(\alpha \beta)^{-d}[\ga,M_{\frac{l}{\alpha}} T_{\frac{n}{\beta}}\g]=\delta_{l0} \delta_{n0}I_{B(\BH)}$ for $l,n \in \BZ^d.$
\end{thm}
\begin{proof}
The implication $(ii)\Rightarrow(i)$ is trivial consequence of Janssen's representation.\\
For the converse $(i)\Rightarrow(ii)$, assume that $S_{\g,\ga}=I$. Let $\f,\h \in L^{\infty}(Q_{1/\beta},\BH)$ and let $l,m \in \BZ^d$ be arbitrary. Then
\begin{eqnarray*}
\delta_{lm}[\f,\h] &=& \delta_{lm} \int_{\BR^d}\f(x)\odot\h(x)dx\\
&=& \delta_{lm} \int_{\BR^d} S_{\g,\ga} \f(x)\odot\h(x)dx\\
&=& \int_{\BR^d} S_{\g,\ga} T_{\frac{l}{\beta}}\f(x)\odot T_{\frac{m}{\beta}}\h(x)dx\\
&=& \beta^{-d} \int_{\BR^d} \sum_{n \in \BZ^d} G_n(x) \left(T_{\frac{n+l}{\beta}}\f(x) \right)\odot T_{\frac{m}{\beta}} \h(x)dx\\
&=& \beta^{-d} \int_{\BR^d} G_{m-l}(x) \left(T_{\frac{m}{\beta}}\f(x) \right)\odot T_{\frac{m}{\beta}} \h(x)dx\\
&=& \beta^{-d} \int_{\BR^d} (T_{-\frac{m}{\beta}}G_{m-l}(x))\f(x) \odot \h(x)dx\\
&=& \beta^{-d} [(T_{-\frac{m}{\beta}}G_{m-l})(\f), \h]
\end{eqnarray*}
By density this identity extends for all $\f,\h \in L^2(Q_{1/ \beta},\BH)$. Thus $\beta^{-d}G_{m-l}(x+\frac{m}{\beta})=\delta_{lm}I_{B(\BH)}$ for almost all $x \in Q_{1/\beta}$. Varying $l,m \in \BZ^d$ it follows that $G_n(x)=\beta^{d}\delta_{n0} I_{B(\BH)}$, for almost all $x \in \BR^d$. Therefore by (\ref{eq33}) and uniqueness of Fourier coefficients we have $(\alpha \beta)^{-d}[\ga,M_{\frac{l}{\alpha}} T_{\frac{n}{\beta}}\g]=\delta_{l0} \delta_{n0}I_{B(\BH)}$ .
\end{proof}
Now we are in a position to prove Theorem \ref{th1}.\\
{\bf Proof of Theorem \ref{th1}:}
$(a)$ 
Given that $\g \in \L1$ and $1 \leq p, q \leq \infty.$ Let $\f \in \LPN$.
 Viewing $\f$ as a vector in $W(L^1, L^{\infty}_{1/w})$, the series defining $\m_k$ converges in $L^1(Q_{\alpha},\BH)$. However we show that the series $\m_k$ converges actually converges unconditionally in $L^p(Q_{1/\beta},\BH)$ (weakly if $p = \infty$). For each fixed $k$, we consider
\begin{eqnarray}
 &&\int_{Q_{1/ \beta}} | \langle  \sum_{n \in \BZ^d} (\h \odot T_{\alpha k +\frac{n}{\beta} }\g(x))T_{\frac{n}{\beta}}\f(x), \hp(x) \rangle_{\BH} | dx \nonumber \\
& =&  \sum_{n \in \BZ^d} \int_{Q_{\alpha}} |\langle \f(x), T_{\alpha k}\g(x)\rangle_{\BH} \langle \h, \hp(x)\rangle_{\BH} |  T_{\alpha k +\alpha n} \chi_{Q_{\alpha}}(x) dx \nonumber\\
& \leq&  \sum_{n \in \BZ^d} \Vert \g \cdot T_{\alpha n} \chi_{Q_{\alpha}} \Vert_{L^{\infty}(\BR^d, \BH)} \Vert \f \cdot T_{\alpha k +\alpha n} \chi_{Q_{\alpha}}  \Vert_p  K^{1/p'}_{\alpha \beta} \Vert \hp \Vert_{p',Q_{1/\beta} } \frac{C_v v(\alpha k +\alpha n)w(\alpha n)}{v(\alpha k)} \nonumber\\
& =&  C_v K^{1/p'}_{\alpha \beta} \Vert \hp \Vert_{p',Q_{1/\beta} } \frac{1}{v(\alpha k)} \sum_{n \in \BZ^d} \Vert \g \cdot T_{\alpha n} \chi_{Q_{\alpha}} \Vert_{L^{\infty}(\BR^d, \BH)} w(\alpha n)  \Vert \f \cdot T_{\alpha k +\alpha n} \chi_{Q_{\alpha}}  \Vert_p v(\alpha k +\alpha n)\nonumber.
\end{eqnarray}
Taking the supremum in over $\hp$ with unit norm we get,
\begin{eqnarray*}\label{eq10}
& \Vert \m_k \Vert_{p, Q_{1/ \beta}} \leq \beta^{-d} C_v K^{1/p'}_{\alpha \beta} \frac{1}{v(\alpha k)} \sum_{n \in \BZ^d} \Vert \g \cdot T_{\alpha n} \chi_{Q_{\alpha}} \Vert_{L^{\infty}(\BR^d, \BH)} w(\alpha n) \Vert \f \cdot T_{\alpha k +\alpha n} \chi_{Q_{\alpha}}  \Vert_p v(\alpha k +\alpha n),
\end{eqnarray*} where $ K_{\alpha \beta}=\max\limits_{k \in \BZ^d} \#\{ \ell\in \BZ^d : |(\frac{\ell}{\beta}+Q_{1/ \beta}) \cap (\alpha k +Q_{\alpha})|>0 \}. $
This shows the convergence of the series defining $\m_k$ in $L^p(Q_{1/\beta},\BH)$. Note that,
\begin{eqnarray*}
\hat{\m}_k(n) &=& \beta^d \int_{Q_{1/ \beta}} \m_k(x) e^{- 2 \pi i \beta \langle n,x \rangle} dx \\
&=& \int_{Q_{1/ \beta}} \sum_{m \in \BZ^d} (\h \odot T_{\alpha k +\frac{m}{\beta} }\g(x))T_{\frac{m}{\beta}}\f(x) e^{- 2 \pi i \beta \langle n,x \rangle} dx \\
&=& C_{\g, \h}\f(k,n)
\end{eqnarray*}
Now we show that $C_{\g, \h}$ is a bounded mapping of $\LPN$ into $\SP$. Define $r(k)=\Vert \m_k \Vert_{p, Q_{1/ \beta}},\; k \in \BZ^d $. We show the sequence  $ r(k)\in \ell^q_{\tilde{v}}$ which would imply $C_{\g, \h}\f \in \SP$. Take the sequence $a \in \ell^{q'}_{1/ \tilde{v}}$. Then we have
\begin{align}\label{eq11}
|\langle r,a \rangle| & \leq \sum_{k \in \BZ^d} \Vert \m_k \Vert_{p, Q_{1/ \beta}} |a(k)| \nonumber \\
& \leq \beta^{-d} C_v K^{1/p'}_{\alpha \beta}  \sum_{n \in \BZ^d} \Vert \g \cdot T_{\alpha n} \chi_{Q_{\alpha}} \Vert_{L^{\infty}(\BR^d, \BH)} w(\alpha n) \nonumber \\
& \;\;\;\;\;\;\; \left( \sum_{k \in \BZ^d} \Vert \f \cdot T_{\alpha k +\alpha n} \chi_{Q_{\alpha}}  \Vert^q_p {v(\alpha k +\alpha n)}^q \right)^{1/q} \left( \sum_{k \in \BZ^d} |a(k)|^{q'} \frac{1}{{v(\alpha k)}^{q'}} \right)^{1/{q'}} \nonumber \\
& \leq \beta^{-d} C_v K^{1/p'}_{\alpha \beta} \Vert \g \Vert_{\L1} \Vert \f \Vert_{\LPN} \Vert a \Vert_{\ell^{q'}_{1/ \tilde{v}}}.
\end{align}
Taking the supremum  over sequences $a$ with unit norm in (\ref{eq11}) we get
$ \Vert C_{\g, \h}\f \Vert_{\SP}=\Vert r \Vert_{\ell^q_{\tilde{v}}} \leq \beta^{-d} C_v K^{1/p'}_{\alpha \beta} \Vert \g \Vert_{\L1} \Vert \f \Vert_{\LPN}.$
Hence $ C_{\g, \h}$ is a bounded mapping of $\LPN$ into $\SP$. This proves $(a)$.

$(b)$ 
Let $1 \leq p,q < \infty.$ Given $d \in \SP$, we have $\sum\limits_{k \in \BZ^d} \Vert \m_k \Vert^q_{p,Q_{1/ \beta}} \tilde{v}(k)^q < \infty.$ That means for every $\varepsilon>0$, there exists a finite set $F_0$ such that
\begin{equation}\label{eq5}
 \sum\limits_{k \notin F} \Vert \m_k \Vert^q_{p,Q_{1/ \beta}} \tilde{v}(k)^q < \varepsilon^q, \;\; \forall \; \mbox{finite}~ F \supset F_0.
\end{equation}
Since $1/v$ is also an $w$-moderate weight, for any $\hp \in \LPQ$, we have
\begin{eqnarray}
&& \sum_{k \notin F} | \langle \langle \m_k(\cdot),\h \rangle_{\BH} T_{\alpha k} \g , \hp \rangle | \nonumber \\
& \leq&   \sum_{k \notin F} \int_{\BR^d} | \langle \langle \m_k(x),\h \rangle_{\BH} T_{\alpha k} \g(x) , \hp(x) \rangle_{\BH} | dx \nonumber  \\
&=& \sum_{k \notin F} \sum_{n \in \BZ^d} \int_{Q_{\alpha}} |\langle \m_k(x),\h \rangle_{\BH} \langle T_{\alpha k} \g(x) , \hp(x) \rangle_{\BH} | T_{\alpha n + \alpha k} \chi_{Q_{\alpha}}(x) dx  \nonumber \\
& \leq& \sum_{k \notin F} \sum_{n \in \BZ^d} \Vert T_{\alpha k} \g \cdot T_{\alpha n + \alpha k} \chi_{Q_{\alpha}} \Vert_{L^{\infty}(\BR^d, \BH)} \Vert \m_k \Vert_{p, \alpha n + \alpha k+ Q_{\alpha} } \Vert \hp \cdot T_{\alpha n + \alpha k} \chi_{Q_{\alpha}} \Vert_{p'} \;\; \frac{v(\alpha k)}{v(\alpha n + \alpha k- \alpha n)} \nonumber \\
& \leq& \sum_{n \in \BZ^d} \Vert \g \cdot T_{\alpha n} \chi_{Q_{\alpha}} \Vert_{L^{\infty}(\BR^d, \BH)}\sum_{k \notin F} K^{1/p}_{\alpha \beta} \Vert \m_k \Vert_{p, Q_{1/ \beta} } \Vert \hp \cdot T_{\alpha n + \alpha k} \chi_{Q_{\alpha}} \Vert_{p'} \;\; \frac{C_v v(\alpha k)w(\alpha n)}{v(\alpha n + \alpha k)}\nonumber.
\end{eqnarray}
Using (\ref{eq5}), we get $\sum_{k \notin F} | \langle \langle \m_k(\cdot),\h \rangle_{\BH} T_{\alpha k} \g , \hp \rangle | \leq \epsilon C_v K^{1/p}_{\alpha \beta}  \Vert \g \Vert_{\L1} \Vert \hp \Vert_{\LPQ}.$
Therefore,
$R_{\g,\h}d=\sum\limits_{k \in \BZ^d} \langle \m_k(\cdot),\h \rangle_{\BH} T_{ak}\g $ converges unconditionally. Replacing $F$ by $\BZ^d$, we get
\begin{eqnarray} \label{eq7}
 |\langle R_{\g,\h}d, \hp \rangle|
 & \leq & \sum_{k \in \BZ^d} | \langle \langle \m_k(\cdot),\h \rangle_{\BH} T_{\alpha k} \g , \hp \rangle | \nonumber \\
& \leq & C_v K^{1/p}_{\alpha \beta} \Vert \g \Vert_{\L1} \Vert d \Vert_{\SP} \Vert \hp \Vert_{\LPQ}.\nonumber
\end{eqnarray}
Thus
\begin{eqnarray} \label{eq8}
\Vert R_{\g,\h}d \Vert_{\LPN}  \leq  C_v K^{1/p}_{\alpha \beta} \Vert \g \Vert_{\L1} \Vert d \Vert_{\SP},
\end{eqnarray}
This completes the proof for the case $1 \leq p,q < \infty .$
 Since $\LPQ$ is the K\"othe dual of $\LPN$, a similar argument as in (\ref{eq5}), (\ref{eq7}) will imply the  convergence  of $R_{\g,\h}$ in the weak topology when $p= \infty $ or $q= \infty $ and the estimate $\Vert R_{\g,\h}d \Vert_{\LPN}$ as in (\ref{eq8}). This proves part $(b)$.

$(c)$ Now we show the frame operator  $R_{\ga,\h}C_{\g,\h}$  admits Walnut's representation on $\LPN.$\\
Given $\g, \ga \in \L1$ and $1 \leq p,q \leq \infty .$
Notice that for a $w$-moderate weight $v$ and $\f \in \LPN$, $1 \leq p,q \leq \infty$ and for each $n \in \BZ^d$, $\alpha>0$ we have
\begin{equation}\label{eqt}
\Vert T_{\alpha n} \f \Vert_{\LPN} \leq C_v w(\alpha n) \Vert \f \Vert_{\LPN}
\end{equation}
Replacing $\alpha$ by $1/ \beta$ in (\ref{eqt}) we get,
\[ \Vert T_{\frac{n}{ \beta}} \f \Vert_{\LPN} \leq C_v w(\frac{n}{ \beta}) \Vert \f \Vert_{\LPN}. \]
Therefore, for $\f \in \LPN$ consider
\begin{eqnarray*}
\sum_{n \in \BZ^d} \Vert G_n \left(T_{\frac{n}{\beta}}\f \right) \Vert_{\LPN} & \leq & \sum_{n \in \BZ^d} \Vert G_n \Vert_{L^{\infty}(\BR^d, B(\BH))} \Vert T_{\frac{n}{\beta}}\f \Vert_{\LPN} \\
& \leq & C_v \Vert \f \Vert_{\LPN} \sum_{n \in \BZ^d} \Vert G_n \Vert_{L^{\infty}(\BR^d, B(\BH))} w(\frac{n}{ \beta})\\
& \leq & C C_v \Vert \f \Vert_{\LPN} \Vert \g \Vert_{\L1} \Vert \ga \Vert_{\L1},
\end{eqnarray*}
by Lemma \ref{lem4}. Therefore the series $\sum\limits_{n \in \BZ^d}  G_n\left(T_{\frac{n}{\beta}}\f \right) $ converges absolutely in $\LPN$.\\
Now for fixed $\f \in \LPN$, define $\m_k$ such that $c_{\g, \h} \f(k,n) =\hat{\m}_k(n)$ and $\hp \in \LPQ$ we have
\begin{eqnarray*}
 \langle R_{\ga , \h}C_{\g, \h} \f ,\hp \rangle
&=& \sum_{k \in \BZ^d} \langle \langle \m_k(\cdot),\h \rangle_{\BH} T_{ak}\ga ,\hp \rangle \\
&=& \sum_{k \in \BZ^d} \int_{\BR^d} \langle \langle \m_k(x),\h \rangle_{\BH} T_{ak}\ga(x) ,\hp(x) \rangle_{\BH} dx \\
&=& \beta^{-d} \sum_{k \in \BZ^d} \int_{\BR^d} \sum_{n \in \BZ^d} \langle \langle (\h \odot T_{\alpha k +\frac{n}{\beta} }\g(x))T_{\frac{n}{\beta}}\f(x) ,\h \rangle_{\BH} T_{ak}\ga(x) ,\hp(x) \rangle_{\BH} dx \\
&=& \beta^{-d} \sum_{k \in \BZ^d} \int_{\BR^d} \sum_{n \in \BZ^d} \langle \langle \langle T_{\frac{n}{\beta}}\f(x), T_{\alpha k +\frac{n}{\beta} }\g(x) \rangle_{\BH} \h ,\h \rangle_{\BH} T_{ak}\ga(x) ,\hp(x) \rangle_{\BH} dx \\
&=& \beta^{-d} \sum_{k \in \BZ^d} \int_{\BR^d} \sum_{n \in \BZ^d} \langle \langle T_{\frac{n}{\beta}}\f(x), T_{\alpha k +\frac{n}{\beta} }\g(x) \rangle_{\BH} T_{ak}\ga(x) ,\hp(x) \rangle_{\BH} dx \\
&=& \beta^{-d} \sum_{n \in \BZ^d} \int_{\BR^d} \sum_{k \in \BZ^d} \langle (T_{ak}\ga(x) \odot T_{\alpha k +\frac{n}{\beta} }\g(x))T_{\frac{n}{\beta}}\f(x) ,\hp(x) \rangle_{\BH} dx \\
&=& \beta^{-d} \sum_{n \in \BZ^d}  \langle G_n\left(T_{\frac{n}{\beta}}\f \right), \hp \rangle.
\end{eqnarray*}
The interchanges of integration and summation can be justified by Lemma \ref{lem4} and Fubini's Theorem. This proves (c).$ \hfill\square$

\subsection{Convergence of Gabor expansions}

\begin{prop}\label{pr3}
Let $\alpha,\beta>0$ and $1<p<\infty, \; 1 \leq q < \infty$ and $\g, \ga \in \L1$. Let $v$ be an $w$-moderate weight on $\BR^d$.  If $\h \in \BH$ with unit norm and $\mathcal{G}(g, \alpha, \beta)$ is a Gabor frame for $L^2(\BR^d,\BH)$ with dual window $\ga$ then
\begin{enumerate}
\item[(a)] the partial sums
\[ S_{K,N}d=\sum_{| k |\leq K} \sum_{|n| \leq N} \langle d_{kn}, \h \rangle_{\BH} M_{\beta n}T_{\alpha k}\g, \;\;\; K,N>0, \]
converge to $R_{\g,\h}d$ in the norm of $\LPN$, where $d \in \SP.$
\item[(b)] the partial sums of the Gabor expansion \[ S_{K,N}(C_{\g,\h}\f)=\sum_{| k |\leq K} \sum_{|n|\leq N} \langle \f,M_{\beta n}T_{\alpha k}\g \rangle M_{\beta n}T_{\alpha k}\ga  \]
converge to $\f$ in the norm of $\LPN$ and in the $\sigma(\LPN,\LPQ)$-topology for $1 \leq p,q \leq \infty$.
\end{enumerate}
\end{prop}
\begin{proof}
(a) Let $\varepsilon>0$ be given. Write $R_{\g,\h}d-S_{K,N}d=(R_{\g,\h}d-S_{K_0}d)+(S_{K_0}d-S_{K_0,N}d)+(S_{K_0,N}d-S_{K,N}d),$ where $S_{K_0}d=\sum_{|k|\leq K_0}\langle m_k(\cdot), h\rangle_\BH T_{\alpha k}\g.$
Using the boundedness of the operator $R_{\g,\h}:\SP \rightarrow \LPN$ we have
\begin{eqnarray}\label{eq17}
\Vert R_{\g,\h}d-S_{K_0}d \Vert_{\LPN}
&=& \Vert R_{\g,\h}\Vert \left( \sum_{|k| > K_0} \Vert \m_k \Vert^q_{p,Q_{1/ \beta}} \tilde{v}(k)^q \right)^{1/q} \nonumber \\
&\leq & \Vert R_{\g,\h}\Vert \varepsilon.
\end{eqnarray}
Again repeating a similar argument as above and (\ref{ew4}) we have \begin{eqnarray}\label{eq18}
\Vert S_{K_0}d-S_{K_0,N}d \Vert_{\LPN} &\leq & \Vert R_{\g,\h}\Vert \left( \sum_{|k| \leq K_0} \Vert \m_k - S_N \m_k \Vert^q_{p, Q_{1/\beta}} \tilde{v}(k)^q \right)^{1/q} \nonumber \\
&\leq & \Vert R_{\g,\h}\Vert \varepsilon,
\end{eqnarray}
where $S_N \m_k=\sum_{|n|\leq N} d_{kn}e^{2 \pi i \beta \langle n,x \rangle} $ converging to $\m_k$ in the norm of $L^p(Q_{1/ \beta}, \BH)$, cf. \cite{ru}.
Finally \begin{eqnarray}\label{eq19}
\Vert S_{K_0,N}d-S_{K,N}d \Vert_{\LPN}&=& \Vert R_{\g,\h}\Vert \left( \sum_{K_0<|k| \leq K} \Vert S_N \m_k \Vert^q_{p,Q_{1/ \beta}} \tilde{v}(k)^q \right)^{1/q} \nonumber \\
&\leq & C_1 \Vert R_{\g,\h}\Vert  \left( \sum_{K_0<|k| \leq K} \Vert \m_k \Vert^q_{p,Q_{1/ \beta}} \tilde{v}(k)^q \right)^{1/q}  \nonumber \\
&\leq & C_1 \Vert R_{\g,\h}\Vert \varepsilon,
\end{eqnarray}
where for each $k \in \BZ^d$ we can find a $C_1>0$ satisfies
\begin{equation}\label{eq13}
\sup_{N>0}\Vert S_N \m_k \Vert_{p, Q_{1/\beta}} \leq C_1 \Vert \m_k \Vert_{p, Q_{1/\beta}}.
\end{equation}

Choose $K_0,N_0\in\mathbb{N}$ such that the inequalities (\ref{eq17})-(\ref{eq19}) are satisfied for all $K>K_0,~N>N_0$ and $\Vert R_{\g,\h}d-S_{K,N}d \Vert_{\LPN} \leq (2+C_1) \Vert R_{\g,\h}\Vert \varepsilon$, which completes the proof of part (a).\\
(b) By  (\ref{*}) $R_{\g,\h}C_{\g,\h}=I$ on $L^2(\BR^d,\BH)$. This  would imply $G_n(x)=\delta_{n0}\beta^dI_{B(\BH)}$ for almost all $x\in\mathbb{R}^d.$ Using (\ref{eq3}) and the conclusion of the previous part together proves (b).
\end{proof}
 The Walnut's representation for superframe operator and the multi-window Gabor frame operator can be obtained by choosing an appropriate Hilbert space in Theorem \ref{th1}. However, we list out some consequences of Theorem \ref{th1} and Proposition \ref{pr3} in the following remark.
\begin{rmk}\label{r}

(i) If $\BH=\BC$ then the rank one operator $x\odot y$ turns out to the point-wise product $xy$ and all the above results for the $\BH$-valued Gabor frame viz. Walnut's representation of $\BH$-valued Gabor frame operator, convergence of Gabor expansions, etc., coincides with the results for the scalar valued Gabor frames (see \cite{fei, oko, wal, weis}).

(ii) If $\BH=\BC^n$ then the  $\BH$-valued Gabor frame is the super Gabor frame (see \cite{lyu}). The  Gabor expansions for the Gabor super-frames on vector valued amalgam spaces also converges by Proposition \ref{pr3}.

(iii) Let $(\mathbb{H}_1,\langle\cdot,\cdot\rangle_1),(\mathbb{H}_2,\langle\cdot,\cdot\rangle_2),\cdots,(\mathbb{H}_r,\langle\cdot,\cdot\rangle_r)$ be $r$ Hilbert spaces. If $\mathbb{H}= \bigoplus_{i=1}^r\mathbb{H}_i$ (i.e $\mathbb{H}$ is the direct sum of $r$ Hilbert spaces) then $\mathbb{H}$ is also a Hilbert space with respect to the inner product $\langle x,y \rangle=\displaystyle\sum_{i=1}^r\langle x_i,y_i \rangle$ where $x=\displaystyle\oplus_{i=1}^r x_i,~y=\oplus_{i=1}^r y_i,\,x,y\in\mathbb{H}, x_i,y_i\in \mathbb{H}_i, i=1,2\cdots,r.$ If $f:\mathbb{R}^d\to \mathbb{H} $ then $f(x)=f_1(x) \bigoplus f_2(x)\bigoplus\cdots \bigoplus f_r(x)$ with $f_i(x)\in \mathbb{H}_i,~i=1,2\cdots,r.$ Note that $f\in W_\mathbb{H}(L^p,L_w^q)\Leftrightarrow f_i\in W_{\mathbb{H}_i}(L^p,L_w^q)$ for all $i=1,2,\cdots,r.$ The frame operator of the Gabor system on $W_\mathbb{H}(L^p,L_w^q)$ with respect to a single lattice $\Lambda=\alpha\mathbb{Z}\times\beta\mathbb{Z}$ is given by
\begin{eqnarray*}
S_{\g,\ga}\f&=&\beta^{-d} \sum_{n \in \BZ^d} G_n\left( T_{\frac{n}{\beta}} \f \right)\\
&=&\beta^{-d} \sum_{n \in \BZ^d}\sum_{k \in \BZ^d} T_{\alpha k} \ga(x) \odot T_{\alpha k + \frac{n}{\beta}} \g(x)\left( T_{\frac{n}{\beta}}\f\right)\\
&=& \beta^{-d} \sum_{n \in \BZ^d}\sum_{k \in \BZ^d} T_{\alpha k} \left(\bigoplus_{i=1}^r\ga_i(x)\right) \odot T_{\alpha k + \frac{n}{\beta}} \left(\bigoplus_{i=1}^r\g_i(x)\right)\left( T_{\frac{n}{\beta}}\left(\bigoplus_{i=1}^r\f_i\right)\right)\\
&=& \beta^{-d} \sum_{n \in \BZ^d}\sum_{k \in \BZ^d}\sum_{i=1}^r \left(T_{\alpha k}\ga_i(x)\odot T_{\alpha k + \frac{n}{\beta}}\g_i(x)\right)\left( T_{\frac{n}{\beta}}\f_i\right)\\&=& \sum_{i=1}^rS_{\g_i,\ga_i}\f_i,
\end{eqnarray*} where $\g(x)=\bigoplus_{i=1}^r\g_i(x),\,\ga(x)=\bigoplus_{i=1}^r\ga_i(x)$ and $S_{\g_i,\ga_i}$ is the frame operator of the Gabor system on $W_{\mathbb{H}_i}(L^p,L_w^q)$.
 If $\mathbb{H}_1=\mathbb{H}_2\cdots=\mathbb{H}_r=\mathbb{C}$ then the $\BH-$valued Gabor frame turns out to Gabor superframe.

(iv) Let $\Lambda=\Lambda^1 \times ... \times \Lambda^r$ be the Cartesian
product of separable lattices $\Lambda^i=\alpha_i \BZ^d \times \beta_i \BZ^d$ and let $\g_1,...,\g_r,\ga_1,...,\ga_r \in W_{\mathbb{H}_i}(L^{\infty}, L^1_w)$. Suppose the collection $\mathcal{G}_i(\g_i,\alpha_i,\beta_i)$ is a frame for $L^2(\mathbb{R}^d,\mathbb{H}_i)$ with the corresponding frame operator $S_{\g_i,\ga_i}^{\Lambda^i}$. As in the pervious set up (as in (iii)) we show that frame operator associated  with the Gabor system on $W_{\mathbb{H}}(L^p,L_w^q)$  is the sum of frame operator associated  with the Gabor systems on $W_{\mathbb{H}_i}(L^p,L_w^q)$. In this case we consider cartesian product of separable lattices $\Lambda^i=\alpha_i \BZ^d \times \beta_i \BZ^d,~ i=1, 2, \cdots, r$ instead of a single lattice $\alpha \BZ^d \times \beta \BZ^d$. For $1\leq p,q\leq\infty$ the space $\SP$ turns out to be $S_{\tilde{v}}^{p,q}(\mathbb{H}_1)\times S_{\tilde{v}}^{p,q}(\mathbb{H}_2)\times\cdots\times S_{\tilde{v}}^{p,q}(\mathbb{H}_r)$ with the norm $$\|d\|_{\SP}=\|(d^1,d^2\cdots,d^r)\|_{\SP}=\sum_{i=1}^r\|d^i\|_{S_{\tilde{v}}^{p,q}(\mathbb{H}_r)},$$ where ${S_{\tilde{v}}}^{p,q}(\mathbb{H}_i)$is defined as in Definition \ref{s} with respect to the lattice $\Lambda^i$ and the Hilbert space $H_i$. For $x\in\mathbb{R}^d, k,n\in\mathbb{Z}^d,~ \alpha=(\alpha_1,\alpha_2\cdots,\alpha_r)$ and $\beta=(\beta_1,\beta_2\cdots,\beta_r)$ with $\alpha_i>0,\beta_i>0$  the translation operator $T_{\alpha k}$ and the modulation operator $M_{\beta n}$ is defined as $T_{\alpha k}\f(x)=\bigoplus_{i=1}^rT_{\alpha_i k}\f_i(x)$ and $M_{\beta n}\f(x)=\bigoplus_{i=1}^rM_{\beta_i n}\f_i(x)$, where $\f(x)=\bigoplus_{i=1}^r\f_i(x)$. The frame operator of the Gabor system on $W_\mathbb{H}(L^p,L_w^q)$ is given by \begin{eqnarray*}S_{\g,\ga}^{\Lambda}\f&=&\beta^{-d} \sum_{n \in \BZ^d} G_n\left( T_{\frac{n}{\beta}} \f \right)\\&=&\beta^{-d} \sum_{n \in \BZ^d}\sum_{k \in \BZ^d} T_{\alpha k} \ga(x) \odot T_{\alpha k + \frac{n}{\beta}} \g(x)\left( T_{\frac{n}{\beta}}\f\right)\\
    &=& \beta^{-d} \sum_{n \in \BZ^d}\sum_{k \in \BZ^d}  \left(\bigoplus_{i=1}^rT_{\alpha_i k}\ga_i(x)\right) \odot  \left(\bigoplus_{i=1}^rT_{\alpha_i k + \frac{n}{\beta_i}}\g_i(x)\right)\left( \left(\bigoplus_{i=1}^rT_{\frac{n}{\beta_i}}\f_i\right)\right)\\&=& \beta^{-d} \sum_{n \in \BZ^d}\sum_{k \in \BZ^d}\sum_{i=1}^r \left(T_{\alpha_i k}\ga_i(x)\odot T_{\alpha_i k + \frac{n}{\beta_i}}\g_i(x)\right)\left( T_{\frac{n}{\beta_i}}\f_i\right)\\&=& \sum_{i=1}^rS_{\g_i,\ga_i}^{\Lambda^i}\f_i,
        \end{eqnarray*} where $\beta^{-d}=\prod_{i=1}^r\beta_i^{-d}, \g(x)=\bigoplus_{i=1}^r\g_i(x),\,\ga(x)=\bigoplus_{i=1}^r\ga_i(x)$ and $S_{\g_i,\ga_i}^{\Lambda^i}$ is the frame operator of the Gabor system on $W_{\mathbb{H}_i}(L^p,L_w^q)$ with respect to the lattice $\Lambda^i$.

        If $\mathbb{H}_1=\mathbb{H}_2\cdots=\mathbb{H}_r=\mathbb{C}$ and $\f(x)=f(x)\bigoplus f(x)\bigoplus\cdots\bigoplus f(x)$ (r-times tensor product of $f(x)$ with itself) where $f\in W(L^p,L_w^q)$ then the $\BH-$valued Gabor frame turns out to ``multi-window Gabor frame".
\end{rmk}
\section{The algebra of $\BH$-valued $L^{\infty}$-weighted shifts}
\subsection{$\BH$-valued $L^{\infty}$-weighted shifts}
In this section we construct a Banach algebra, based on the structure of Walnut's representation for the frame operator in (\ref{eq3}) and develop necessary tools for invertibility of the frame operator on $\LPN.$  For an admissible weight function $w$ we construct a Banach $*-$algebra $\mathcal{A}_w$ of weighted shift operators in $B(L^p(\mathbb{R}^d,\mathbb{H}))$ and identify with $AP^p_w(\rho)$, the class of $\rho$-almost periodic elements, having $w$-summable Fourier coefficients. Finally we prove the spectral invariance theorem on $AP^p_w(\rho)$ which assures spectral invariance property on $\mathcal{A}_w$.
 Let us start with a definition of multiplication operator on $B(L^p(\BR^d,\BH))$.
\begin{defn}
Let $\phi \in L^{\infty}(\BR^d,B(\BH))$, then we define the multiplication operator $T_{\phi}:L^p(\BR^d,\BH)\rightarrow L^p(\BR^d,\BH)$ defined by $(T_{\phi}f)(x)=\phi(x)(f(x)), \; x \in \BR^d.$
\end{defn}
Note that $T_{\phi}$ is linear, bounded and $\Vert T_{\phi} \Vert=\Vert \phi \Vert_{L^{\infty}(\BR^d,B(\BH))} .$ For an admissible weight $w$  define 
$\mathcal{A}_w:=\{\mathcal{M}=(\m_x)_{x \in \BR^d}\in L^{\infty}(\BR^d,B(\BH)) :\displaystyle\sum_{x \in \BR^d} \Vert \m_x \Vert_{L^{\infty}(\BR^d,B(\BH))}w(x)< +\infty\},$ with norm
$\|\mathcal{M}\|_{\mathcal{A}_w}=\displaystyle\sum_{x \in \BR^d} \Vert \m_x \Vert_{L^{\infty}(\BR^d,B(\BH))}w(x)< +\infty.$


  If the family $\mathcal{M}=(\m_x)_{x \in \BR^d}\in\mathcal{A}_w$, then $\mathcal{M}=(\m_x)_{x \in \BR^d}$ has countable support. The identification of $\mathcal{A}_w$ with the subclass of bounded operators on $L^2(\BR^d,\BH)$ is as follows: Given $(\m_x)_{x \in \BR^d} \in \mathcal{A}_w$, define the operator $T:L^p(\BR^d,\BH)\rightarrow L^p(\BR^d,\BH)$ by $T(\f)=\sum\limits_{x \in \BR^d}\m_x(T_x\f).$ Clearly $T$ is well defined, linear and bounded on all $L^p(\BR^d, \BH),\; 1\leq p \leq \infty $ (by using admissibility of $w$). The identification $\f \mapsto \sum\limits_{x \in \BR^d}\m_x(T_x\f)$ maps $\mathcal{A}_w$ into a closed subspace of $B(L^p(\BR^d,\BH))$. We write $\mathcal{M} \in \mathcal{A}_w $ means
$\mathcal{M}=\sum\limits_{x \in \BR^d} \m_x(T_x),  \;(\m_x)_{x \in \BR^d} \in \ell^1_w(\BR^d,L^{\infty}(\BR^d,B(\BH))).$

If we endow $\mathcal{A}_w$ with the product and involution inherited from $B(L^2(\BR^d,\BH))$ then $\mathcal{A}_w$ is a Banach *-algebra which embeds continuously into $B(L^2(\BR^d,\BH))$:  for $(\m_x)_{x \in \BR^d},({\bf n}_x)_{x \in \BR^d} \in \mathcal{A}_w$, define
$\left( \sum_{x \in \BR^d} \m_x(T_x) \right) \cdot \left( \sum_{x \in \BR^d}{\bf n}_x(T_x) \right)=\sum_{x \in \BR^d}\left( \sum_{y \in \BR^d}\m_y {\bf n}_{x-y}(\cdot -y) \right)(T_x), $
and the involution by
$ \left( \sum_{x \in \BR^d} \m_x(T_x) \right)^*=\sum_{x \in \BR^d} \overline{\m_x(\cdot +x)}(T_{-x})=\sum_{x \in \BR^d} \overline{\m_{-x}(\cdot -x)}(T_{x}). $
Notice that the identification of families in $\mathcal{A}_w$ and operators on $B(L^p(\BR^d,\BH))$ is one-to-one. We simply write $\mathcal{A}_w \subset B(L^p(\BR^d,\BH)) $ and we treat members of $\mathcal{A}_w$ as operators on $L^p(\BR^d,\BH)$. The following result for $\m \in L^{\infty}(\BR^d,B(\BH))$ plays a crucial role in proving the spectral invariance theorem.
\begin{lem}\label{lem6}
 For $\m \in L^{\infty}(\BR^d,B(\BH))$ and $x,w \in \BR^d$. The following relation hold.
\begin{equation}\label{eq31}
M_{w}\m(T_xM_{-w})=e^{2 \pi i \langle w,x \rangle}\m(T_x)
\end{equation}
\end{lem}
\begin{proof}
The proof of this result is trivial if $\mathbb{H}=\mathbb{C}.$ Otherwise, for each $y\in \mathbb{R}^d,~m(y)$ is a linear bounded operator on $\mathbb{H}$, which we view as an infinite matrix with scalar entries and prove the lemma.
Since $\BH$ is separable Hilbert space and for $y \in \BR^d$, $\m(y) \in B(\BH)$ can be written as $\m(y)u=\m(y)\left(\sum_{n}\langle u,e_n \rangle_{\BH}e_n \right)=\sum_{n}\langle u,e_n \rangle_{\BH}\m(y)e_n=\sum_{n,j}\langle u,e_n \rangle_{\BH}a_{nj}(y)e_j,$ where $u \in \BH$ and $(e_n)_n$ is orthonormal basis for $\BH$. Now for $\f \in L^p(\BR^d,\BH)$
\begin{eqnarray*}
M_w\m(y)(T_xM_{-w}\f(y))&=& e^{2 \pi i \langle w,y  \rangle} \sum_{n,j}\langle T_xM_{-w}\f(y),e_n \rangle_{\BH}a_{nj}(y)e_j\\
&=& e^{2 \pi i \langle w,x  \rangle} \sum_{n,j}\langle T_x\f(y),e_n \rangle_{\BH}a_{nj}(y)e_j\\
&=& e^{2 \pi i \langle w,x  \rangle} \m(y)(T_x\f(y))
\end{eqnarray*}
\end{proof}
\begin{prop}\label{pr1}
Let $1\leq p,q \leq \infty $ and let $v$ be a $w$-moderate weight. Then
\begin{enumerate}
\item[(a)] the algebra $\mathcal{A}_w$ is continuously embedded in $B(\LPN) $. 
\item[(b)] for every $\mathcal{M} \in \mathcal{A}_w, \f \in \LPN $ and $\g \in \LPQ$, $ \langle \mathcal{M}(\f), \g \rangle = \langle \f, \mathcal{M}^*(\g) \rangle $.
Moreover, the operator $\mathcal{M}$ is continuous in $\sigma(\LPN, \LPQ )$-topology.
\end{enumerate}
\end{prop}
\begin{proof}
The proof of the proposition follows by a straight forward calculation.
\end{proof}
\subsection{Spectral invariance}
In order to identify the class $\mathcal{A}_w \subset B(L^p(\BR^d,\BH))$ with $AP^p_w(\rho)$ the class of $\rho$-almost periodic elements having $w$-summable Fourier series. We need the following definitions and necessary theory and apply Theorem 3.2 of \cite{bal4}.

Let $1 \leq p \leq \infty, y \in \BR^d$ and $\mathcal{M} \in B(L^p(\BR^d,\BH))$. Define $\rho(y)\mathcal{M}:=M_y\mathcal{M}M_{-y}$. Then, \[ \rho(y)\mathcal{M}\f(x)=e^{2 \pi i \langle y,x \rangle} \mathcal{M}(\g(x)), \; \mathrm{where} \;\; \g(x)=e^{- 2 \pi i \langle y,x \rangle}\f(x). \]
Clearly $\rho$ is a representation of $\BR^d$ on the Banach space $B(L^p(\BR^d,\BH))$. For each $y \in \BR^d$, $\rho(y)$ is an algebra automorphism and an isometry.
\begin{defn}
A continuous map $Y:\BR^d \rightarrow B(L^p(\BR^d,\BH))$ is \textit{almost-periodic in the sense of Bohr} if for every $\varepsilon >0$ there is a compact set $K=K_{\varepsilon}\subset \BR^d$ such that for all $x \in \BR^d$ \[ (x+K)\cap\{ y \in \BR^d:\Vert Y(g+y)-Y(g)\Vert<\varepsilon,\; \forall g \in \BR^d \}\neq \emptyset \]
\end{defn}
Then $Y$ extends uniquely to a continuous map of the Bohr compactification $\hat{R}^d_c$ of $\BR^d$, denoted by $Y$. Thus $Y:\hat{R}^d_c \rightarrow B(L^p(\BR^d,\BH)) $, where $\hat{R}^d_c$ represents the
topological dual group (i.e. the group of characters) of $\BR^d$ when $\BR^d$ is endowed with the discrete topology. The normalized Haar measure on $\hat{R}^d_c$ is denoted by $\overline{\mu}(dy).$

For each $\mathcal{M} \in B(L^p(\BR^d,\BH))$, we consider the map $\widehat{\mathcal{M}}:\BR^d \rightarrow B(L^p(\BR^d,\BH))$ defined by
\begin{equation}\label{eq24}
 \widehat{\mathcal{M}}(y):=\rho(y)\mathcal{M}= M_y\mathcal{M}M_{-y}.
\end{equation}
If the map $\widehat{\mathcal{M}}$ is continuous and almost-periodic in the sense of Bohr then the operator $\mathcal{M}\in B(L^p(\BR^d,\BH))$ is called $\rho$-\textit{almost periodic}. For every $\rho$-almost periodic operator $\mathcal{M}$, the function $\widehat{\mathcal{M}}$ admits a $B(L^p(\BR^d,\BH))$-valued Fourier series,
\begin{equation}\label{eq25}
\widehat{\mathcal{M}}(y)\sim \sum_{x \in \BR^d}e^{2 \pi i \langle y,x \rangle} C_x(\mathcal{M}),\;\;\;(y \in \BR^d).
\end{equation}
The coefficients $C_x(\mathcal{M}) \in B(L^p(\BR^d,\BH))$ in (\ref{eq25}) are uniquely determined by $\mathcal{M}$ via
\begin{equation}\label{eq26}
C_x(\mathcal{M})=\int_{\hat{R}^d_c} \widehat{\mathcal{M}}(y) e^{-2 \pi i \langle y,x \rangle}\overline{\mu}(dy)=\lim_{T \rightarrow \infty}\frac{1}{(2T)^d}\int_{[-T,T]^d}\widehat{\mathcal{M}}(y) e^{-2 \pi i \langle y,x \rangle}dy
\end{equation}
and satisfy
\begin{equation}\label{eq27}
\rho(y)C_x(\mathcal{M})=e^{2 \pi i \langle y,x \rangle} C_x(\mathcal{M}).
\end{equation}
Therefore, they are eigenvectors of $\rho$ (see \cite{bal4} for details).

Within the class of $\rho$-almost periodic operators consider $AP^p_w(\rho)$, the subclass of those operators for which the Fourier series in (\ref{eq25}) is $w$-summable, where $w$ is an admissible weight. More precisely, a $\rho$-almost periodic operator $\mathcal{M}$ belongs to $AP^p_w(\rho)$ if its Fourier coefficients with respect to $\rho$ satisfy
\begin{equation}\label{eq28}
\Vert \mathcal{M} \Vert_{AP^p_w(\rho)}:=\sum_{x \in \BR^d} \Vert C_x(\mathcal{M}) \Vert_{B(L^p(\BR^d,\BH))}w(x) < \infty.
\end{equation}
Since $w$ is submultiplicative, for  $\mathcal{M}\in AP^p_w(\rho)$ the series
\begin{equation}\label{eq29}
\widehat{\mathcal{M}}(y)= \sum_{x \in \BR^d}e^{2 \pi i \langle y,x \rangle} C_x(\mathcal{M}),\;\;\;y \in \BR^d,
\end{equation} converges absolutely to $\widehat{\mathcal{M}}(y)$ on $B(L^p(\BR^d,\BH))$,
where each $C_x \in B(L^p(\BR^d,\BH)) $ satisfies (\ref{eq26}) and hence (\ref{eq27}). In particular, for $y = 0$, each $\mathcal{M} \in AP^p_w(\rho)$ can be written as
\begin{equation}\label{eq30}
\mathcal{M}= \sum_{x \in \BR^d} C_x(\mathcal{M}).
\end{equation}
Conversely, if $\mathcal{M}$ is given by (\ref{eq30}) with the coefficients $C_x$ satisfying (\ref{eq28}) and (\ref{eq27}), it follows from the theory of almost-periodic series that $\mathcal{M} \in AP^p_w(\rho)$ and $C_x$ satisfy (\ref{eq26}). Now we are in a position to establish connection between $\mathcal{A}_w$ and $AP^p_w(\rho)$ and prove spectral invariance result for $\mathcal{A}_w$. For that we first characterize the eigenvectors $C_x$ of the representation $\rho$.
\begin{lem}
For any $1 \leq p \leq \infty$ and any $\m \in L^{\infty}(\BR^d,B(\BH))$ and $x \in \BR^d$, $C_x=\m(T_x)$ is an eigenvector of $\rho: \BR^d \rightarrow B(B(L^p(\BR^d,\BH))) $. For $1 \leq p < \infty$ these are the only eigenvectors.
\end{lem}
\begin{proof}
If $C_x=\m(T_x)$, then by (\ref{eq31}), it satisfies (\ref{eq27}).\\
For $1\leq p<\infty$, take $C_x \in B(L^p(\BR^d,\BH))$ satisfying (\ref{eq27}). Further the relation (\ref{eq31}) implies $\rho(y)C_x(T_{-x})=M_y C_x M_{-y}(T_{-x})=C_x(T_{-x}).$
Therefore $C_x (T_{-x}M_y)=M_yC_x(T_{-x})$, which in turn imply that $C_x (T_{-x})$ must be a multiplication operator $m$. So $C_x =\ m(T_x).$
\end{proof}
For $p = \infty$ there are eigenvectors of $\rho$ which may not of the form $\m(T_x).$ An example of such an eigenvector is given in (\cite{kur}, Section 5.1.11) for the case $\BH=\BC$. Hence $AP^p_w(\rho)$ consists of all the operators $\mathcal{M}=\sum\limits_{x \in \BR^d} C_x$,
with $C_x$ satisfying (\ref{eq28}) and (\ref{eq27}). The previous lemma says that for $1 \leq p < \infty$ an operator $C_x$ satisfies (\ref{eq27}) if and only if it is of the form $C_x = \m(T_x)$, for some
function $\m \in L^{\infty}(\BR^d,B(\BH))$. Note that $\Vert C_x \Vert_{B(L^p(\BR^d,\BH))}=\Vert \m \Vert_{L^{\infty}(\BR^d,B(\BH))}$ and thus $$\|\mathcal{M}\|_{\mathcal{A}_w}=\displaystyle\sum_{x \in \BR^d} \Vert \m_x \Vert_{L^{\infty}(\BR^d,B(\BH))}w(x)=\sum_{x \in \BR^d} \Vert C_x(\mathcal{M}) \Vert_{B(L^p(\BR^d,\BH))}w(x)=\Vert \mathcal{M} \Vert_{AP^p_w(\rho)}.$$ This gives the identification of $\mathcal{A}_w$ with $AP^p_w(\rho).$
\begin{prop}\label{pr2}
For $p \in [1, \infty)$ the class $\mathcal{A}_w \subset B(L^p(\BR^d,\BH))$ coincides with $AP^p_w(\rho)$.
\end{prop}
For $p = \infty$, the two classes are different. Now we are in a position to prove that the algebra $\mathcal{A}_w$ is spectral with in the class of bounded operators on $L^p(\BR^d,\BH)$. This means if an operator from $\mathcal{A}_w$ is invertible on $L^p(\BR^d,\BH)$ then the inverse operator necessarily belongs to $\mathcal{A}_w$. In other words invertibility in the bigger algebra implies the invertibility in the smaller algebra. To prove these kind of results one makes use of Wiener's $1/f$ lemma or its several versions. We resort to recent Wiener type result on non-commutative almost periodic Fourier series (\cite{bal4}, Theorem 3.2) to obtain the following theorem.

\begin{thm}\label{th8}
Let $w$ be an admissible weight and  $\mathcal{M}=\sum\limits_{x \in \BR^d}\m_x(T_x)\in \mathcal{A}_w$ be an invertible operator on $B(L^p(\BR^d,\BH))$ for some $p\in[1,\infty]$ then $\mathcal{M}^{-1} \in \mathcal{A}_w.$
\end{thm}
\begin{proof}
For $1 \leq p < \infty $ the result follows from Proposition \ref{pr2} and Theorem 3.2 in \cite{bal4}.

For $p=\infty$, take \[ \mathcal{M}=\sum_{x \in \BR^d}\m_x(T_x) \in \mathcal{A}_w \subset B(L^{\infty}(\BR^d,\BH))\] with $\sum\limits_{x \in \BR^d} \Vert \m_x \Vert_{L^{\infty}(\BR^d,B(\BH))}w(x)< \infty$. Define \[ \mathcal{N}=\sum_{x \in \BR^d}(T_x(\m_{-x}))(T_x)=\sum_{x \in \BR^d}\m_{-x}(\cdot -x)(T_x) \in \mathcal{A}_w \subset B(L^{1}(\BR^d,\BH)). \] Since $\Vert T_x(\m_{-x}) \Vert_{L^{\infty}(\BR^d,B(\BH))}=\Vert \m_{-x} \Vert_{L^{\infty}(\BR^d,B(\BH))}$ $\mathcal{N}$ is well defined and a straight forward calculation shows that $\mathcal{M}$ is the transpose (Banach adjoint) of $\mathcal{N}$. Therefore $\mathcal{N}$ is invertible when $\mathcal{M}$ is invertible. Since $\mathcal{A}_w$ is spectral in $B(L^1(\BR^d,\BH))$ we get $\mathcal{M}^{-1}=(\mathcal{N}^{-1})' \in \mathcal{A}_w$. That means $\mathcal{M}^{-1}=\sum\limits_{x \in \BR^d}{\bf n}_x(T_x)$ for some ${\bf n}_x \in L^{\infty}(\BR^d,B(\BH)) $ such that $\sum\limits_{x \in \BR^d} \Vert {\bf n}_x \Vert_{L^{\infty}(\BR^d,B(\BH))}w(x)< \infty$.
\end{proof}

\section{$\BH$-valued Dual Gabor frames on amalgam spaces}

\subsection{Invertibility of the frame operators}
\begin{thm}\label{th3}
Let $w$ be an admissible weight, $v$ be $w$-moderate weight and $\g \in \L1$. Suppose that the Gabor system $\mathcal{G}(\g, \alpha, \beta)=\{M_{\beta n}T_{\alpha k} \g : \; k,n \in \BZ^d \}$ is a frame for $L^2(\BR^d,\BH)$ with frame operator $S_{\g}$.
Then inverse operator $S^{-1}_{\g}: \LPN \rightarrow \LPN$, $1 \leq p, q \leq \infty$ is continuous both in $\sigma(\LPN, \LPQ)$ and the norm topologies.
\end{thm}
\begin{proof}
As a consequence of the Walnut's representation in Theorem \ref{th1} the frame operator $S_{\g}$ belongs to the algebra $\mathcal{A}_w.$ As $S_{\g}$ is invertible in $L^2(\BR^d, \BH),$ Theorem \ref{th8} implies that $S^{-1}_{\g} \in \mathcal{A}_w.$ Since $\mathcal{A}_w$ is continuously embedded in $B(\LPN)$, by Proposition \ref{pr1} $S^{-1}_{\g} \in B(\LPN)$. Hence the theorem follows.
\end{proof}

Let $C_0(\BR^d,B(\BH))$ be the subspace formed by the functions of $L^\infty(\BR^d,B(\BH))$ that are continuous. The next corollary shows the continuity of the dual generator provided the window function is continuous.

\begin{cor}
If the window function $g$ is continuous then under the assumptions of Theorem \ref{th3} the dual window $\tilde{\g}=S^{-1}_{\g}(\g)$ is also continuous.
\end{cor}
\begin{proof}
Let
\[\tilde{\mathcal{A}_w}= \Big\{ \mathcal{M}=(\m_x)_{x \in \BR^d}\in \ell^1_w(\BR^d,C_0(\BR^d,B(\BH)))\;| \sum_{x \in \BR^d} \Vert \m_x \Vert_{C_0(\BR^d,B(\BH))}w(x)< \infty \Big\}. \]
Then $\tilde{\mathcal{A}_w} \subset \mathcal{A}_w \subset B(L^p(\BR^d,\BH))$. If $\g \in W_{\BH}(C_0,L^1_w)$ then $S_{\g}=\beta^{-d}\sum\limits_{n \in \BZ^d}G_n(T_{\frac{n}{\beta}}) \in \tilde{\mathcal{A}_w}$. Since $S_{\g}$ is invertible in $B(L^2(\BR^d,\BH))$, applying Theorem \ref{th8} on $\tilde{\mathcal{A}_w}$ we get $S_{\g}^{-1} \in \tilde{\mathcal{A}_w}$.

Let $\g$ be continuous. To show $S_{\g}^{-1}(\g)$ is continuous it is enough to show $S_{\g}$ maps $W_{\BH}(C_0,L^1_w)$ to $W_{\BH}(C_0,L^1_w).$ Let $\f \in W_{\BH}(C_0,L^1_w)$. Since $G_n(x)\left( T_{\frac{n}{\beta}} \f(x) \right)$ is continuous for each $n$, $\sum\limits_{\mathrm{finite}} G_n(x)\left( T_{\frac{n}{\beta}} \f(x) \right)$ is continuous. Again,
\begin{eqnarray*}
\left\Vert \sum\limits_{n \in \BZ^d} G_n(x)\left( T_{\frac{n}{\beta}} \f(x) \right) \right\Vert_{\BH} & \leq & \sum\limits_{n \in \BZ^d} \Vert G_n \Vert_{L^{\infty}(\BR^d,B(\BH))} \Vert T_{\frac{n}{\beta}} \f(x) \Vert_{\BH}\\
&\leq & \sum\limits_{n \in \BZ^d} \Vert G_n \Vert_{L^{\infty}(\BR^d,B(\BH))} \Vert \f \Vert_{\L1} < \infty.
\end{eqnarray*}
So by Weierstrass M-test $\sum\limits_{n \in \BZ^d} G_n(x)\left( T_{\frac{n}{\beta}} \f(x) \right)$ converges uniformly and hence $S_{\g}\f(x)=\beta^{-d}\sum\limits_{n \in \BZ^d} G_n(x)\left( T_{\frac{n}{\beta}} \f(x) \right)$ is continuous.
\end{proof}
\begin{rmk}\label{r1}
(i) By Proposition \ref{pr3} for all $\f \in \LPN$
\begin{equation}\label{eq34}
S_{\g}(\f)=\lim_{K,N \rightarrow \infty}\sum_{\Vert k \Vert_{\infty} \leq K} \sum_{\Vert n \Vert_{\infty} \leq N} \langle \f,M_{\beta n}T_{\alpha k}\g \rangle M_{\beta n}T_{\alpha k}\g ,
\end{equation}
with convergence in the $\sigma(\LPN,\LPQ)$-topology and for $p,q < \infty$ in the norm of $\LPN$. Since $S_{\g}^{-1}\in \mathcal{A}_w$, using Proposition \ref{pr1}(c) and applying $S_{\g}^{-1}$ to both sides of (\ref{eq34}) and we obtain
\begin{eqnarray*}
\f &=& \lim_{K,N \rightarrow \infty}\sum_{\Vert k \Vert_{\infty} \leq K} \sum_{\Vert n \Vert_{\infty} \leq N} \langle \f,M_{\beta n}T_{\alpha k}\g \rangle M_{\beta n}T_{\alpha k} \tilde{\g} \\
&=&  \lim_{K,N \rightarrow \infty}\sum_{\Vert k \Vert_{\infty} \leq K} \sum_{\Vert n \Vert_{\infty} \leq N} \langle \f,M_{\beta n}T_{\alpha k}\tilde{\g} \rangle M_{\beta n}T_{\alpha k}\g. \end{eqnarray*}
 Similarly using Proposition \ref{pr3} we get the convergence in the norm of $\LPN$.

(ii) If $\mathcal{G}(\g,\alpha, \beta)$ is a frame for $L^2(\BR^d,\BH)$ with dual window $\tilde{\g}=S^{-1}_{\g}(\g) \in L^2(\BR^d,\BH)$, then inverse frame operator is given by
\[S^{-1}_{\g,\g}\f=S_{\tilde{\g},\tilde{\g}}\f=\sum_{k \in \BZ^d}\sum_{n \in \BZ^d} \langle \f,T_{\alpha k}M_{\beta n} \tilde{\g} \rangle T_{\alpha k}M_{\beta n}\tilde{\g}. \]

(iii) If $\mathbb{H}=\mathbb{C}$ then Theorem \ref{th3} coincides with Theorem 3.2 of \cite{kri} and Theorem 2 of \cite{weis}. Again if $\mathbb{H}=\mathbb{C}^n$ the invertibility of Gabor superframes on vector valued amalgam spaces is obtained. If we take $\mathbb{H}_1=\mathbb{H}_2\cdots=\mathbb{H}_r=\mathbb{C}$ and $\f(x)=f(x)\bigoplus f(x)\bigoplus\cdots\bigoplus f(x)$ (r-times tensor product of $f(x)$ with itself), where $f\in W(L^p,L_w^q)$ as in Remark \ref{r} (iv) we obtain the invertibility of multi-window Gabor frames on amalgam spaces (see Theorem 6 of \cite{bal3}).

(iv) Since the frame operator $S_{\g}\in\mathcal{A}_w$ and $S_{\g}$ is invertible, by the last line of the proof of Theorem \ref{th8}, $S_{\g}^{-1}$ (as an operator on $L^2(\BR^d,\BH)$), can be expressed as $S_{\g}^{-1}\f(x)=\sum_{k\in\mathbb{Z}^d}G_n(x)(f(x-x_k))$ where the family of points $\{x_k\}$ may not lie in the lattice $\Lambda=\prod_{i=1}^r\alpha_i\mathbb{Z}^d\times\beta_i\mathbb{Z}^d$.

\end{rmk}
\section*{Acknowledgments}
The first author wishes to thank the Ministry of Human Resource Development, India for the  research fellowship and Indian Institute of Technology Guwahati, India for the support provided during the period of this work. The authors would like to thank the referee for many very helpful comments and suggestions that helped us improve the presentation of this paper.


\end{document}